\documentclass[leqno,11pt]{amsart}

\usepackage{bbm}
\usepackage[rgb,dvipsnames]{xcolor}
\usepackage{tikz} 
\usetikzlibrary{decorations.text}
\usetikzlibrary{arrows,arrows.meta,hobby}
\usetikzlibrary{shapes.misc, positioning}

\usepackage{tikz-cd}

\usetikzlibrary{cd}

\usepackage{setspace}
\usepackage[english]{babel}
\usepackage{bussproofs}
\usepackage[utf8]{inputenc}
\usepackage{csquotes}
\usepackage{mathtools}
\usepackage{mathrsfs}
\usepackage{latexsym}
\usepackage{enumitem}
\usepackage{amsthm}
\usepackage{amssymb}
\DeclareMathAlphabet{\mathpzc}{OT1}{pzc}{m}{it}

\usepackage{bbm}

\newtheorem{theorem}{Theorem}[section]
\newtheorem*{theorem*}{Theorem}

\newtheorem{lemma}[theorem]{Lemma}
\newtheorem{proposition}[theorem]{Proposition}
\newtheorem{corollary}[theorem]{Corollary}

\newtheorem{fact}[theorem]{Fact}
\newtheorem{claim}[theorem]{Claim}
\theoremstyle{definition}
\newtheorem{definition}[theorem]{Definition}

\theoremstyle{remark}
\newtheorem{remark}{Remark}

\newtheorem{question}{Question}

\def\hook{\upharpoonright}
\def\forces{\Vdash}

\def\Null{\mathcal N}
\def\ZFC{\mathsf{ZFC}}

\def\PFA{\mathsf{PFA}}
\def\MA{\mathsf{MA}}

\def\BA{\mathsf{BA}}

\def\baire{\omega^\omega}

\def\cantor{2^\omega}

\def\mfb{\mathfrak b}

\def\CH {\mathsf{CH}}
\def\Q{\mathbb Q}
\def\P{\mathbb P}

\def\cc{2^{\aleph_0}}

\def\mfp{\mathfrak{p}}

\def\add{{\rm add}}

\title{Variants of Baumgartner's Axiom for Lipschitz Functions on Baire and Cantor Space}

\author[Switzer]{Corey Bacal Switzer}
\address[C.~B.~Switzer]{Institut f\"{u}r Mathematik, Kurt G\"odel Research Center, Universit\"{a}t Wien, Kolingasse 14-16, 1090 Wien, AUSTRIA}
\email{corey.bacal.switzer@univie.ac.at}

\thanks{\emph{Acknowledgments:} This research was funded in whole or in part by the Austrian Science Fund (FWF) through the following grant: 10.55776/ESP548.}

\begin{document}

\begin{abstract}
   We consider several variants of Baumgartner's axiom for $\aleph_1$-dense sets defined on the Baire and Cantor spaces in terms of Lipschitz functions with respect to the usual metric. A variation of Baumgartner's original argument shows that these variants are consistent. However, unlike in the case of the classical $\BA$, we are able to give many applications for which the corresponding fact for linear orders is open. In particular we show that there are provable implications from the $\baire$ variants to the $\cantor$ variants and that some of these principles imply all the cardinals in the Cicho\'{n}'s diagram are large. We also show, similar to (but not the same as) $\BA$, that none of the Lipschitz variants follow from a large fragment of $\MA$. 
\end{abstract}

\maketitle

\section{Introduction}
The motivation for this work comes from {\em Baumgartner's axiom}, denoted hereafter as $\BA$, which states that all $\aleph_1$-dense sets of reals are isomorphic\footnote{See below for precise definitions.}. It was proved consistent by Baumgartner in \cite{Baum73} and again by different means in \cite{BaumPFA}. This statement is itself a generalization of Cantor's celebrated result on the $\omega$-categoricity of countable dense linear orders. Baumgartner's axiom has a natural topological phrasing which allows one to state parametrized principles $\BA (X)$ for arbitrary topological spaces $X$ ($\BA$ itself is $\BA (\mathbb R)$), see below for precise definitions. These were investigated in \cite{Stepranswatson87}, \cite{BaldwinBeaudoin89} among other places, see also \cite{weakBA} for a different perspective and \cite[Section 3.5]{QsCDH} for related open problems. A curious fact is that the one dimensional case - namely $\BA (\mathbb R)$ -  behaves very differently from both the zero dimensional cases $\BA (\baire)$ and $\BA (\cantor)$ and the higher dimensional cases of $\BA (\mathbb R^n)$ for $n > 1$. In particular, while the non-one dimensional cases all follow from Martin's Axiom (and actually much less see Fact \ref{BAfacts} below), the one dimensional case does not (see \cite{AvrahamShelah81}). The most straightforward explanation for this is that the linear order structure on $\mathbb R$ induces additional constraints on the topology which have no analogue in the non-one dimensional cases. The purpose of this paper is consider similar axioms which appeal to additional, natural combinatorial structure available on other well studied Polish spaces. 


In order to further the discussion we recall the following definition from \cite{Stepranswatson87}.
\begin{definition} \label{topBAdef}
    Let $\kappa$ be a cardinal and $X$ a topological space. The axiom $\BA_\kappa (X)$ states that for all $\kappa$-dense $A, B \subseteq X$ there is an autohomeomorphism $h:X \to X$ so that $h''A = B$. When $\kappa = \aleph_1$ we often drop the subscript. 
\end{definition}
Here recall that a set $A \subseteq  X$ is $\kappa$-dense if its intersection with every non-empty open set has size $\kappa$.


In this paper we consider analogues of $\BA (X)$ for $X$ either $\baire$ or $\cantor$ with additional the additional restriction that the mappings are required to be Lipschitz (and not just arbitrary continuous functions) (but the precise statement is weakened overall). We dub these axioms $\BA_{\rm} (\baire)$ and $\BA_{\rm Lip}(\cantor)$ (see Definition \ref{BAlipDef1} for the precise definition). By mimicking the proof of the consistency of $\BA$ from \cite{Baum73} we show the following. 

\begin{theorem}[See Theorem \ref{consistency} below] \label{consistencyintro}
    Both $\BA_{\rm Lip} (\baire)$ and $\BA_{\rm Lip} (2^\omega)$ can be forced by ccc forcing over a model of $\CH$. They also follow from $\PFA$. 
\end{theorem}
Currently there are no known implications between axioms of the form $\BA (X)$ for different $X$ (barring trivial ones). However in the case of the Lipschitz functions we show the following.
\begin{theorem}
    $\BA_{\rm Lip} (\baire)$ implies $\BA_{\rm Lip} (\cantor)$. \label{implicationintro}
\end{theorem}
Todor\v{c}evi\'c showed in \cite{Todorcevic89} that $\BA$ implies that $\mfb > \aleph_1$ where $\mfb$ is the standard bounding number. Similarly we show the following by a very different argument.
\begin{theorem}
  $\BA_{\rm Lip} (2^\omega)$ (and hence by Theorem \ref{implicationintro}, $\BA_{\rm Lip} (\baire)$ as well) implies that $\mfb > \aleph_1$. \label{bintro}
\end{theorem}

The proof of the consistency of $\BA_{\rm Lip} (\cantor)$ gives a formally stronger statement which we dub $\overline{\BA}_{\rm Lip}(\cantor)$ (and the same for $\baire$). A similar proof as that of Theorem \ref{implicationintro} shows that $\overline{\BA}_{\rm Lip}(\baire)$ implies $\overline{\BA}_{\rm Lip}(\cantor)$. These stronger statements have an effect on cardinal characteristics which is open for $\BA$ and known to be false for $\BA (\cantor)$ and $\BA (\baire)$.
\begin{theorem}
    The axiom $\overline{\BA}_{\rm Lip}(2^\omega)$ (and hence $\overline{\BA}_{\rm Lip}(\baire)$ as well) implies all of the cardinals in the Cicho\'{n} diagram are greater than $\aleph_1$. \label{cichonintro}
\end{theorem}

\begin{corollary}
    $\BA (\cantor)$ does not imply $\overline{\BA}_{\rm Lip}(\cantor)$ and the same for $\baire$.
\end{corollary}

We do not know if the analogue of Avraham and Shelah's theorem from \cite{AvrahamShelah81} that $\BA$ does not follow from $\MA$ holds for our Lipschitz variants, however we have the following partial result.
\begin{theorem}
    $\MA_{\rm \aleph_1} ({\rm Knaster}) +$``every Aronszajn tree is special" does not imply $\BA_{\rm Lip} (\cantor)$.  \label{MAintro}
\end{theorem}

The rest of this paper is organized as follows. In the rest of this section we mention some preliminaries and definitions we will use throughout the paper. In Section 2 we study Lipschitz functions on Baire and Cantor space and record some facts we will need. In Section 3 we introduce the Lipschitz Baumgartner axioms and discuss some basic properties. In particular we prove Theorems \ref{implicationintro} and \ref{bintro}. In Section 4 we prove the consistency of the Lipschitz axioms. In Section 5 we discuss the relation to the null ideal and prove Theorem \ref{cichonintro}. In Section 6 we prove Theorem \ref{MAintro}. Section 7 concludes the paper with a series of open questions. 
\subsection{Preliminaries}

Our notation is standard, conforming to that of the monographs \cite{KenST} and \cite{JechST}. For information of cardinal characteristics we also refer the reader to \cite{BarJu95}. We recall some facts we will need. Let $\kappa$ be a cardinal. Recall that a linear order is $\kappa$-{\em dense} if between any two points there are $\kappa$ many. It is easy to see that a separable linear order is $\aleph_1$-dense if and only if it is isomorphic to a set of reals which is $\aleph_1$-dense in the sense that it has intersection of size $\aleph_1$ with every non-empty open interval (in the real line ordering). Baumgartner's axiom $\BA$ states that all such linear orders are isomorphic. The following is easy to see and motivates Definition \ref{topBAdef} above.

\begin{proposition}[See \cite{Stepranswatson87}]
    $\BA$ holds if and only if $\BA_{\aleph_1}(\mathbb R)$ holds.
\end{proposition}

We need the following cardinal characteristics which we recall here.

\begin{definition}
    \begin{enumerate}
        \item The {\em additivity of the null ideal}, denoted $\add(\Null)$ is the least $\kappa$ for which there are $\kappa$ many Lebesgue measure zero sets of reals whose union is not Lebesgue measure zero\footnote{As is well known, the definition makes no difference if we take $\mathbb R$, $\baire$, $\cantor$ or any other perfect Polish space.}.

        \item Given $f, g \in \baire$ we write $f\leq^*g$ if and only if for all but finitely many $k < \omega$ we have $f(k) \leq g(k)$. In this case we say $g$ {\em eventually dominates} $f$. The {\em bounding number}, denoted $\mfb$ is the least size of an unbounded set $\mathcal B \subseteq \baire$, i.e. a set so that no single $g\in \baire$ eventually dominates every $f \in \mathcal B$.         

        \item If $\mathcal F$ is a family of infinite subsets of natural numbers we say that $\mathcal F$ has the {\em strong finite intersection property} if for every finite subset $\mathcal A \subseteq \mathcal F$ we have that $\bigcap \mathcal A$ is infinite. If $A \subseteq \omega$ is infinite then we say that $A$ is a {\em pseudointersection} of such a family $\mathcal F$ if $A \subseteq^* B$ for all $B \in \mathcal F$ - i.e. for each $B$ there is a $k < \omega$ with $A \setminus k \subseteq B$. In this case we also say $A$ is {\em almost contained in} $B$. The {\em pseudointersection number} $\mfp$ is the least size of a family with the strong finite intersection property but no pseudointersection. 

    \end{enumerate}
\end{definition}

For the reader's convenience we recall the following well known facts, see e.g. \cite{BlassHB}.

\begin{fact}
    The follow inequalities are provable in $\ZFC$. In each case they many be strict.

    \begin{enumerate}
        \item $\mfp \leq \mfb$
        \item $\add(\Null) \leq \mfb$
    \end{enumerate}
    Moreover $\mfp$ and $\add(\Null)$ are independent - both strict inequalities are consistent with $\ZFC$.
\end{fact}

We have the following connection between cardinal characteristics and Baumgartner type axioms.

\begin{fact} \label{BAfacts}
    \begin{enumerate}
        \item For all $\kappa < \mfp$ we have $\BA_\kappa (\cantor)$ and $\BA_\kappa (\baire)$ (\cite{BaldwinBeaudoin89}).
        \item For all $\kappa < \mfp$ and all natural numbers $n > 1$ we have $\BA_\kappa (M^n)$ and $\BA_\kappa (\mathbb R^n)$ hold for any compact $n$-dimensional manifold $M^n$ (\cite{Stepranswatson87}).
        \item $\BA := 
        \BA_{\aleph_1} (\mathbb R)$ does not follow from $\MA + \neg \CH$. In particular $\aleph_1 < \mfp$ does not suffice to conclude $\BA$ (\cite{AvrahamShelah81}).
    \end{enumerate}
\end{fact}

Finally recall that given metric spaces $(X, d_X)$ and $(Y, d_Y)$ an {\em isometry} is a function $i:X \to Y$ which preserves distances i.e. $d_X(x, y) = d_Y(i(x), i(y))$ for every $x, y \in D_X$. A {\em Lipschitz function} (with Lipschitz constant $1$) is a function $l:X \to Y$ so that for all $x, y \in X$ we have $d_X(x, y) \geq d_Y(l(x), l(y))$. Clearly isometries and Lipschitz functions are continuous. 

\section{Some Basics on Lipschitz and Isometries of Baire and Cantor Spaces}

Here we collect some facts about isometries and Lipschitz functions on Baire and Cantor space that we will use throughout the text. Most of these are presumably known though we struggled to find citations so we include (sketches of) proofs for completeness. No attempt has been made to work in further generality than $\baire$ and $\cantor$ as we aim to not distract. To be clear, the metric on $\baire$ is defined by $d(x, y) = \frac{1}{2^k}$ where $k$ is least with $x(k) \neq y(k)$. As $\cantor$ can be viewed as a subspace of $\baire$
we simply define the metric on $\cantor$ as the same\footnote{It is also common in the literature to have the metric be $d(x, y) = \frac{1}{k+1}$ where $k$ is as above. The choice here makes no difference in what we will do. The only point is that the longer sequences agree, the closer they are together.}. In practice this means a function $f:\baire \to \baire$ (or $f:\cantor \to \cantor$) is {\em Lipschitz} (with Lipschitz constant $1$) if for all $k < \omega$ we have that if $x \hook k = y \hook k$ then $f(x) \hook k  = f(y) \hook k$. It is an {\em isometry} if the above is an if and only if i.e. for all $k < \omega$ we have that $x \hook k = y \hook k$ if and only if $f(x) \hook k = f(y) \hook k$.

As we will be working with both $\baire$ and $\cantor$ and many of the basic lemmas hold mutandis mutatis for either of them we will state them for arbitrary spaces of the form $X^\omega$ where $|X| \leq \aleph_0$ and the topology is generated by the basic opens of the form $[s] := \{x \in X^\omega \; | \; s$ is an initial segment of $x\}$ for $s \in X^{<\omega}$. For the rest of this section fix a countable set $X$ as above. Given $s, y \in X^{<\omega}$ denote by $s \lhd t$ if and only if $s$ is an initial segment of $t$. Note that this means that $[s] \supseteq [t]$. Finally recall that a {\em tree} $T \subseteq X^{< \omega}$ is a set which is downwards $\lhd$-closed i.e. if $s \lhd t \in T$ then $s \in T$. The set of {\em branches} through a tree $T$ is the set $\{x \in X^\omega \; | \; \forall n \in \omega \, x \hook n \in T\}$. It's well known that a set $F \subseteq X^\omega$ is closed if and only if it's the set of branches through a tree. 

\begin{definition}
Let $T$ and $S$ be trees. A function $g:T \to S$ is a {\em homomorphism} if the following two conditions hold.
\begin{enumerate}
\item
$g$ is {\em level preserving} i.e. $|s| = |g(s)|$ for all $s \in T$.
\item
$g$ is {\em order preserving} i.e. if $s \lhd t$ then $g(s) \lhd g(t)$.
\end{enumerate}

\end{definition}

Note that if $T$ is a tree and $g:T \to X^{<\omega}$ is a homomorphism then the image of $g$ must be again a tree (of cofinal height). The first observation we need is that homomorphisms of $X^{<\omega}$ (treated as a tree itself) are effectively the same as Lipschitz functions on $X^\omega$.
\begin{lemma} \label{homomorphism representation}
\begin{enumerate}
\item
 If $f:X^\omega \to X^\omega$ is a Lipschitz, then the function $\tilde{f}:X^{<\omega} \to X^{<\omega}$ defined by $\tilde{f}(s) = t$ if and only if there is an $x \in [s]$ so that $f(x) \in [t]$ with $|s| = |t|$ is well defined and is a homomorphism of $X^{<\omega}$. 
\item
If $\tilde{f}:X^{< \omega} \to X^{< \omega}$ is a homomorphism then it induces a unique Lipschitz function $f:X^\omega \to X^\omega$ so that for all $k<\omega$ and all $x \in X^\omega$ we have that $f(x) \hook k = \tilde{f}(x \hook k)$. 
\end{enumerate}
\end{lemma}

\begin{proof}
 Suppose $f$ is Lipschitz. Then for each $x, y \in X^\omega$ and $k < \omega$ we have that $f(x) \hook k = f(y) \hook k$ if $x \hook k = y \hook k$ hence the map $\tilde{f}:X^{<\omega} \to X^{<\omega}$ defined by $\tilde{f}(t) = s$ if and only if there is an $x \in X^\omega$ with $t \supset x$ and $s = f(x) \hook |t|$ is well defined. Moreover it is easily verified that it has the properties claimed.

    Conversely if we have a function $\tilde{f}$ which is order and level preserving then define $f(x)$ to be $\bigcup_{k < \omega} \tilde{f}(x \hook k)$. Again it's easy to see that it has the advertised properties. 
\end{proof}

Next we need to know how to lift Lipschitz functions from from small domains to larger ones.

\begin{lemma}
    Suppose $A \subseteq X^\omega$ and let $f:A \to X^\omega$ be Lipschitz. Then $f$ uniquely lifts to a (Lipschitz) function $\hat{f}:\bar{A} \to X^\omega$ where $\bar{A}$ denotes the closure of $A$. In particular a Lipschitz function on a dense subset of $X^\omega$ lifts uniquely to a Lipschitz function on the entire $X^\omega$.  \label{lifting lipschitz}
\end{lemma}

\begin{proof}
    The point is that the homomorphism determined by $f$ is in fact defined on the tree of finite sequences extending to elements of $A$, whose set of branches is simply the closure of $A$. More concretely, let $f:A \to X^\omega$ be Lipschitz. If $a \in \bar{A}$ then there is a sequence $\{a_n\}_{n < \omega} \subseteq A$ so that for all $n < \omega$ we have that $a \hook n = a_n \hook n$. Fix one and define $\hat{f}(a)$ to be the unique $b$ so that $b\hook n = f(a_n) \hook n$. Note that the Lipschitz property implies that this is well defined since $a_m \hook n = a_n \hook n$ for $m >  n$ and is moreover independent of the choice of $\{a_n\}_{n < \omega}$ since by Lemma \ref{homomorphism representation} we only need the first $n$ bits of $a_n$ to define the map. The last thing to check is that this is indeed a Lipschitz function on $\bar{A}$. To see this, suppose $a, b \in \bar{A}$ and $k < \omega$ is such that $a\hook k = b\hook k$. There are in $A$ elements $a', b' \in A$ so that $a' \hook k = a\hook k$ and $b' \hook k = b\hook k$. We have that $\hat{f} (a)\hook k =f(a')\hook k = f (b') \hook k = \hat{f} (b) \hook k$ as needed. 

\end{proof}

\begin{lemma}
    Suppose $A, B \subseteq X^\omega$ and let $f:A \to B$ be a surjective isometry (and therefore a bijection from $A$ to $B$). Then there is a unique isometry $\hat{f}:\bar{A} \to \bar{B}$. In particular isometric bijections between dense sets lift to isometric (and hence homeomorphic) maps from $X^\omega \to X^\omega$. \label{lifting isometries}
\end{lemma}

\begin{proof}
The map $\hat{f}$ is exactly as defined in the previous paragraph. Let us check that it is a bijection and both it and its inverse are also isometries. We start by checking that it is injective. Suppose that $a, b \in \bar{A}$ are distinct. Let $k < \omega$ be such that $a \hook k \neq b \hook k$. By the way that $\hat{f}$ was defined this means that there are $a_k, b_k \in A$ so that $a_k \hook k = a\hook k$ and $b_k \hook k = b\hook k$. Since $f$ is an isometry we have that $\hat{f}(a) \hook k = f(a_k) \hook k \neq f(b_k) \hook k = \hat{f}(b) \hook k$ and hence $\hat{f}(a) \neq \hat{f}(b)$. 

For surjectivity suppose that $b \in \bar{B}$ and let $\{b_k\}_{k < \omega}\subseteq B$ be a sequence converging to $b$ with $b\hook k = b_k \hook k$. Since $f$ is surjective there are $\{a_k\}_{k < \omega} \subseteq A$ so that $f(a_k) = b_k$ for all $k< \omega$. Now if $j > k$ we have that $b_j \hook  k = b_k \hook k$ and, since $f$ is an isometry therefore $a_j \hook k = a_k\hook k$ and consequently $\{a_k\}_{k < \omega}$ is a Cauchy sequence. Thus there is an $a = \lim_{k \to \infty} a_k$. Now observe though that $\hat{f}(a) = b$ by the way the lift was defined. 

Proving that both $\hat{f}$ and its inverse are isometries is a straightforward generalization of the observation that $\hat{f}$ is Lipschitz in the case of Lemma \ref{lifting lipschitz}. Indeed fix $a, a' \in \bar{A}$ with sequences $\{a_k\}_{k < \omega}$ and $\{a_k'\}_{k < \omega}$ in $A$ so that for all $j < \omega$ we have $a \hook j = a_j \hook j$ and $a' \hook j = a_j'\hook j$. Then for all $k < \omega$ we have $a \hook k = a' \hook k$ if and only if $a_k' \hook k = a' \hook k$ if and only if $f(a_k) \hook k = f(a_k ') \hook k$ if and only if $\hat{f}(a) \hook k = \hat{f}(a') \hook k$. The case of the inverse is identical.
\end{proof}

\begin{remark}
    A subtle though important point here is that we do not in general assume isometries to be surjective. For instance the map $f:\baire \to \baire$ defined so that for all $x \in \baire$ and $n < \omega$ we have $f(x) (n) = x(n) + 1$ is an isometry in our sense but is not surjective since its range is exactly the set of sequences which are never $0$. Nevertheless in the case of $2^\omega$ isometries will be surjective - see Lemma \ref{nowhereisometriesareneversurjective} below.  
\end{remark}

The first real result we'll need is that a variant Cantor's theorem holds $\baire$ and $2^\omega$ when restricted to isometries. 

\begin{lemma}
    Let $A, B \subseteq X^\omega$ be countable and dense. Then there is an isometry $f:X^\omega \to X^\omega$ so that $f``A = B$. 
    \label{CDHisom}
\end{lemma}

\begin{proof}
    This is a back and forth argument. Let $A, B \subseteq X^\omega$ be countable\footnote{We use here for the first time the general assumption that $X$ itself is countable - otherwise $X^\omega$ won't have a countable, dense subset.} and dense and enumerate them as $A = \{a_n\; | \; n < \omega\}$ and $B = \{b_n\; | \; n < \omega\}$. We inductively define isometries $\{f_n\}_{n < \omega}$ so that the following hold for all $n < \omega$.
    \begin{enumerate}
        \item ${\rm dom}({f_n}) \supseteq \{a_0, ..., a_n\}$ and ${\rm range}(f_n) \supseteq \{b_0, ..., b_n\}$.
        \item $f_n$ is an isometry in the sense that for each $a, a' \in {\rm dom}(f_n)$ and all $k < \omega$ we have that $a \hook k = a' \hook k$ if and only if $f(a) \hook k = f(a') \hook k$.
    \end{enumerate}

    Having accomplished this the desired isometry will be the lift of $\bigcup_{n < \omega} f_n$ obtained from Lemma \ref{lifting isometries}. Thus we just need to carry out the induction. The case of $n = 0$ is easy. We simply let $f_0 (a_0) = b_0$, with this being the unique value in the domain. Given $f_n$ satisfying (1) and (2) above we need to find $f_{n+1} \supseteq f_n$ with $a_{n+1}$ in its range and $b_{n+1}$ in its domain. Without loss of generality let us assume that $a_{n+1} \notin {\rm dom}(f_n)$ and $b_{n+1} \notin {\rm range}(f_n)$. If this is not true then we can ignore the ``forth" or ``back" subcases below respectively. We begin with the ``forth" case. Let $k < \omega$ be maximal so that there is an $a \in {\rm dom}(f_n)$ with $a_{n+1} \hook k = a \hook k$ - such exists since the domain of $f_n$ is finite. Let $i < \omega$ so that $a_i \in {\rm dom}(f_n)$ witnessing this $k$ (such an $i$ need not be unique - we will return to this point). Since $B$ is dense, there is a $b \in B$ so that $b \hook k = f(a_i) \hook k$. This is almost good enough but we need to choose $b$ a little more carefully. 

\begin{claim}
There is a $b \in B$ so that for all $j$ with $a_j \in {\rm dom}(f_n)$ if $a \hook k = a_j \hook k$ then $b \hook k = f_n(a_j) \hook k$ and $b(k) \neq f_n(a_j)(k)$.
\end{claim}

Here we return to the point that $a_i$ may not have been unique so that it agreed with $a_{n+1}$ for a maximal length. The claim states that we can find a single $b \in B$ which agrees with not only $f_n(a_i)$ for exactly the first $k$ bits but indeed every such $a_j$ which was maximal in this sense.

\begin{proof}[Proof of Claim]
For short, let us refer to any $a_j \in {\rm dom}(f_n)$ with $a \hook k = a_j \hook k$ as {\em maximally good}. Since $f_n$ is an isometry by hypothesis, we have that for all $a_j, a_l$ which are maximally good that $f_n(a_j) \hook k = f_n(a_l) \hook k$. Consider now the set of all $x \in X$ so that $f_n(a_j)(k) = x$ for some maximally good $a_j$. If this does not encompass all of $X$ then we can find a $y \in X$ so that for {\em no} maximally good $a_j$ do we have $f_n(a_j)(k) = y$ and $b$ can be chosen as any element of $B$ intersecting the open set $[f_n(a_j) \hook k^\frown y]$ (which again exists by density). Such a $b$ would witness the claim. Otherwise every $x \in X$ is such that $f_n(a_j)(k) = x$ for some maximally good $a_j$. In this case $X$ must be finite and, since $f_n$ is assumed to be an isometry we have that for every $x \in X$ there is a maximally good $a_j$ with $a_j(k) = x$. But this is a contradiction to the choice of $k$ since then $a_{n+1}(k) = a_j(k)$ for some maximally good $a_j$ so in particular $k$ was not maximal.
\end{proof}

Let us now fix a $b \in B$ as in the claim and check that $f_{n} \cup \{\langle a_{n+1}, b\rangle\}$ is an isometry. We need to show that for all $a_j \in {\rm dom}(f_n)$ and all $k$ we have $a_j \hook k = a_{n+1} \hook k$ if and only if $f_n(a_j) \hook k = f_n(a_{n+1}) \hook k$. There are two cases.  

    \noindent \underline{Case 1}: $a_j$ is not maximally good. Then there is a largest $l < k$ so that $a_j \hook l \neq a_{n+1} \hook l$. We have that for all $m < l$ that $a_j \hook m = a_{n+1} \hook m = a_i \hook m$ (where $a_i$ is the maximally good element we fixed in the beginning of the proof). But then, by the fact that $f_n$ is an isometry we have $f_n (a_j) \hook m = f_n (a_i) \hook m = b\hook m$ by the choice of $b$. However $b \hook l = f_n(a_i) \hook l\neq f_n(a_j)\hook l$ again since $f_n$ is an isometry and the choice of $b$. Putting all this together we get that $a\hook m = a_j\hook m$ if and only if $b \hook m = f_n(a_j)\hook j$, as needed.

    \noindent \underline{Case 2}: $a_j$ is maximally good. The case is somewhat similar but reduces to the choice of $b$.

    The ``back" part of the argument is entirely symmetric and left to the reader. 
    
\end{proof}

It is tempting to try and generalize the statement of Lemma \ref{CDHisom} to $\aleph_1$-dense sets, however this is flat out inconsistent. 
\begin{proposition}
    Let $\kappa \leq \cc$ be uncountable. There are two $\kappa$-dense subsets of $X^\omega$ so that no uncountable subset of either one can be isometrically mapped into the other. \label{counterexample}
\end{proposition}

\begin{proof}
     Fix a distinguished element $0_X \in X$. For each $s \in X^{<\omega}$ let $O_s$ be a $\kappa$ sized set of elements of $[s]$ so that for all $k \geq|s|$ if $x \in O_s$ is such that $x(k) \neq 0_X$ then $k$ is odd. Similarly let $E_s$ be the same with ``odd" in the previous paragraph replaced by ``even"\footnote{Here, $O$ is used for odd and $E$ for even.}. Note that for each $s \in X^{<\omega}$ if $x, y \in O_s$ then the first place that they differ is even. Similarly if $x, y \in E_s$ then the first place they differ is odd. Let $O = \bigcup_{s \in X^{<\omega}} O_s$ and $E = \bigcup_{s \in X^{<\omega}} E_s$. Clearly these are $\kappa$-dense. Suppose that $A \subseteq O$ is uncountable (the case of $A \subseteq E$ is symmetric). If $f:A \to E$ is any function at all there is some $s \in \omega^{<\omega}$ by pigeonhole so that $f^{-1}[E_s]$ is uncountable. Applying pigeonhole again, there is a $t \in \omega^{<\omega}$ so that $f^{-1}[E_s] \cap O_t$ is uncountable and in particular has size $2$. But if $x, y \in O_t$ then the first $k$ at which they differ is odd while if $f(x), f(y) \in E_s$ then the first place they differ is even hence $f$ is not an isometry. 
\end{proof}

We note that the above shows that any function between $E$ and $O$ as defined above is not only not an isometry when restricted to any uncountable set but moreover is {\em nowhere an isometry} in the sense that it cannot be an isometry on open set. We will not need it but also just note that these facts will remain true in any forcing extension that doesn't collapse $\omega_1$. Finally the following fact will be useful.
\begin{lemma} 
    Let $f:2^\omega \to 2^\omega$ be Lipschitz. The following are equivalent.

    \begin{enumerate}
        \item $f$ is surjective.
        \item $f$ is an isometry.
    \end{enumerate}

    In particular Lipschitz maps on Cantor space are isometries if and only if they are homeomorphisms. \label{nowhereisometriesareneversurjective}
\end{lemma}

Note that the last line follows from the fact that isometries are always injective and continuous bijections between compact Hausdorff spaces are always homeomorphisms. We remark that the above holds not just for $\cantor$ but in fact any space of the form $\Pi_{n < \omega} f(n)$ for any give $f:\omega \to \omega$ - i.e. the set of possibilities for each coordinate must just be finite. 

\begin{proof}
    Fix $f:2^\omega \to 2^\omega$, Lipschitz. We begin with $(1)$ implies $(2)$. Suppose $f$ is surjective and, towards a contradiction assume there are $x, y \in 2^\omega$ and $k \in \omega$ so that $x(k) \neq y(k)$ but $f(x) \hook k+1 = f(y) \hook k+1 := t$. In other words $x$ and $y$ witness that $f$ is not an isometry as $f$ maps them properly closer together. In this case then the map $\tilde{f}$ from Lemma \ref{homomorphism representation} maps both $x \hook k+1$ and $y \hook k+1$ (which are assumed to be distinct) onto this same $t$. This implies however that $\tilde{f} \hook 2^{k+1}$ is not injective and hence not surjective (since the set is finite) and therefore ${\rm range}(f)$ is disjoint from some basic open $[s]$ with $s \in 2^{k+1}$ and in particular is not surjective. 

    Now we prove $\neg (1)$ implies $\neg (2)$. Since $f$ is continuous, the image of $f$ is compact and hence closed in $2^\omega$. As such, if $f$ is not surjective, then there is some $k < \omega$ and some finite sequence $s \in 2^k$ so that ${\rm range} (f) \cap [s] = \emptyset$. It follows then that the map $\tilde{f} \hook 2^k$ is not a surjection and hence there are distinct sequences $t_0, t_1 \in 2^k$ so that $\tilde{f}(t_0) = \tilde{f}(t_1)$. But then if $x_0 \supseteq t_0$ and $x_1 \supseteq t_1$ we have that $d(x, y) \geq \frac{1}{2^k}$ while $d(f(x), f(y)) < \frac{1}{2^k}$ so $f$ is not an isometry.
\end{proof}

\begin{remark}
    The above holds first of all for any finite $X$ (and not just $2$). Moreover it holds when restricted to any open set and therefore we can actually show that if $f:\cantor \to \cantor$ is Lipschitz and nowhere an isometry then the image of $f$ is closed nowhere dense.
\end{remark}

\section{The Lipschitz Baumgartner Axioms}

We now introduce the axioms we will study for the rest of the paper. There are four - two for Baire space and two for Cantor space. Here is the first pair, the ``basic Lipschitz $\BA$".

\begin{definition} \label{BAlipDef1}
    Let $Z$ be either $\baire$ or $\cantor$. The axiom $\BA_{\rm Lip}(Z)$ states that for all $A, B \subseteq Z$ which are $\aleph_1$-dense there is a Lipschitz $f:Z \to Z$ so that $f$ injects $A$ into $B$. 
\end{definition}

The proof of the consistency of these axioms gives something slightly stronger and these stronger versions turn out to be useful.
\begin{definition} \label{BAlipDef2}
    The {\em strong Lipschitz Baumgartner axiom for} $\baire$, denoted $\overline{\BA}_{\rm Lip}(\baire)$ is the statement that for all $A, B \subseteq \baire$ which are $\aleph_1$-dense there is a Lipschitz $f:\baire \to \baire$ so that $f$ bijects $A$ onto $B$. 
\end{definition}
By Lemma \ref{nowhereisometriesareneversurjective} the above statement with $\baire$ replaced by $2^\omega$ is in fact inconsistent by Lemma \ref{nowhereisometriesareneversurjective} applied to the counterexample from Proposition \ref{counterexample}. However we have the following slightly weaker version.
\begin{definition}
    The {\em strong Lipschitz Baumgartner axiom for} $\cantor$, denoted $\overline{\BA}_{\rm Lip}(\cantor)$ is the statement that for all $A, B \subseteq \cantor$ which are $\aleph_1$-dense there are countably many Lipschitz $f_n:\cantor \to \cantor$, $n < \omega$ so that $B = \bigcup_{n < \omega} f_n``A$. 
\end{definition}
In the next section we will prove these axioms are consistent. For now we delay this though and study some basic consequences which we hope will give the reader a feel for the axioms and will be useful in proving the consistency. The first of these is that there are provable implications between them the four axioms - note the corresponding implications in the case of the classical Baumgartner axioms is unknown.

\begin{theorem} \label{implications}
 The following implications are provable in $\ZFC$.
 \begin{enumerate}
     \item $\BA_{\rm Lip}(\baire)$ implies $\BA_{\rm Lip}(\cantor)$.
     \item $\overline{\BA}_{\rm Lip}(\baire)$ implies $\overline{\BA}_{\rm Lip}(\cantor)$.
 \end{enumerate}
\end{theorem}

I do not know currently if any of the arrows can be reversed. Questions relating to this will be asked in the final section of this paper. The interest in Theorem \ref{implications} is that each $\baire$ version implies its corresponding $\cantor$ one. To the best of my knowledge this is the first example of a ``Baumgartner type axiom" for one space implying the same for another (in a non trivial sense). The figure below gives the full table of implications.
 
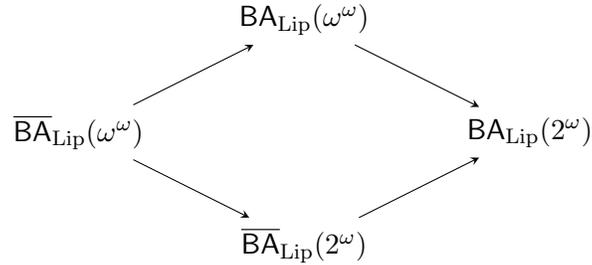
\begin{figure}[h]\label{figure}
\centering
  \begin{tikzpicture}[scale=1.5,xscale=2]
     \draw (0,1) node (BASB) {$\overline{\BA}_{\rm Lip}(\baire)$}
           
	(1,0) node (BASC){$\overline{\BA}_{\rm Lip}(\cantor)$}
	(1,2) node (BAB){$\BA_{\rm Lip}(\baire)$}
	(2,1) node (BAC) {$\BA_{\rm Lip}(\cantor)$}
	
           ;
     \draw[->,>=stealth]
            (BASB) edge (BAB)
            (BASB) edge (BASC)
            (BASC) edge (BAC)
            (BAB) edge (BAC)
            
            ;
  \end{tikzpicture}
  
 \end{figure}


We prove each item of Theorem \ref{implications} separately as lemmas.

\begin{lemma}
    The axiom $\BA_{\rm Lip}(\baire)$ implies $\BA_{\rm Lip}(\cantor)$. \label{firstlemma}
\end{lemma}

\begin{proof}
    Assume $\BA_{\rm Lip}(\baire)$ and let $A, B \subseteq 2^\omega$ be $\aleph_1$-dense. Let $A' \subseteq \baire$ be $\aleph_1$-dense extending $A$ (which is possible since $A$ is subset of $\baire$ trivially since it is a subset of $2^\omega$). We want to apply $\BA_{\rm Lip}(\baire)$ to $A'$ and a $B'$ we will define below in terms of $B$. Towards this we need the following.

    \begin{claim}
        There is a $B' \subseteq \baire$ which is $\aleph_1$-dense and so that the parity map $\pi:B' \to B$ defined by $\pi(b)(k) = b(k) \, {\rm mod} \, 2$ is a bijection.
    \end{claim}

    \begin{proof}[Proof of Claim]
        In the proof it will be convenient to extend the parity map to elements of $\omega^{<\omega}$ in the obvious way i.e. for $s \in \omega^{<\omega}$ we let $\pi(s) (k) = s(k) \, {\rm mod} \, 2$ for each $k \in {\rm dom}(s)$. Let $B$ be enumerated as $\{b_i\; | \; i \in \omega_1\}$. Inductively define $\{b'_i\; | \; i \in\omega_1\}$ as follows:  first fix $b'_0 = b_0$. Now suppose that $\gamma < \omega_1$ and for all $\alpha < \gamma$ we have defined $b'_\alpha$ so that $\pi (b'_\alpha) = b_\alpha$. For each $s \in \omega^{<\omega}$ define the {\em order type} of $s$ to be the order type of the set of $\{\beta < \gamma\; | \; s \subseteq b'_\beta\}$. Let $s_\gamma \in \omega^{<\omega}$ be the lexicographic least of minimal order type so that $\pi(s_\gamma) \subseteq b_\gamma$. Now pick an $x \in \baire$ so that $\pi (x) = b_\gamma$. Let $b'_\gamma = s_\gamma {}^\frown x \hook [|s_\gamma|, \omega)$. Clearly $\pi (b'_\gamma) = b_\gamma$ and $b'_\gamma \neq b'_\alpha$ for any $\alpha < \gamma$ - since if $b'_\gamma = b'_\alpha$ then their images under $\pi$ would be the same, which is false by assumption. This ensures that $\pi:B' \to B$ is a bijection where $B' = \{b'_\gamma \; | \; \gamma \in \omega_1\}$. 

        It remains to check that $B'$ is $\aleph_1$-dense in $\baire$. Suppose not and let $s \in \omega^{<\omega}$ be lexicographic least so that $B' \cap [s]$ is countable. Let $\delta < \omega_1$ so that $s \nsubseteq b_\alpha$ for any countable $\alpha > \delta$. By $\aleph_1$-density of $B$ in $2^\omega$ there are $\aleph_1$-many $\gamma > \delta$ so that $\pi (s) \subseteq b_\gamma$ thus by the construction $s$ could not have been the $t$ which was lexicographic least of minimal order type with $\pi (t) \subseteq b_\gamma$. By Applying Fodor we get a stationary subset of this set of $\gamma$ for which the same $t$ was chosen. Moreover, this $t$ must have order type less than $\delta$ (since otherwise it could not be of minimal order type) but $t$ was chosen uncountably many times, which is a contradiction.

    \end{proof}

    Now let $f:A' \to B'$ be an injective Lipschitz function - the existence of such is guaranteed by $\BA_{\rm Lip}(\baire)$. Note that $\pi:\baire \to \cantor$ is also Lipschitz in the sense that if $x, y \in \baire$ and $k < \omega$ so that $x \hook k = y \hook k$ then $\pi (x) \hook k = \pi (y) \hook k$. Consider the restriction $\pi \circ f \hook A: A \to \pi (B')$. Since $\pi (B') = B$. This is a Lipschitz function from $A$ to $B$ which is moreover injective since $\pi$ is a bijection on $B'$. Thus this is the required map to witness $\BA_{\rm Lip}(2^\omega)$. 
\end{proof}

Similarly we have the following:
\begin{lemma}
    The axiom $\overline{\BA}_{\rm Lip}(\baire)$ implies $\overline{\BA}_{\rm Lip}(\cantor)$.
\end{lemma}

\begin{proof}
    Assume $\overline{\BA}_{\rm Lip}(\baire)$. As in the proof of Lemma \ref{firstlemma} let $A, B \subseteq 2^\omega$. Let $B'$ be defined exactly as in the proof of Lemma \ref{firstlemma} - i.e. $B' \subseteq \baire$ is $\aleph_1$-dense and $\pi:B' \to B$ is a bijection. For each $s \in {\omega}^{<\omega}$ let $A_s = \{s^\frown a\; | \; a \in A\}$. Clearly $A' = \bigcup_{s \in \omega^{<\omega}} A_s$ is $\aleph_1$-dense. Applying $\overline{\BA}_{\rm Lip}(\baire)$, fix $f:A' \to B'$ a Lipschitz bijection. For each $s \in \omega^{<\omega}$ let $g_s:A \to A_s$ be the obvious map i.e. $g_s (a) = s^\frown a$. Note that for each $s \in \omega^{<\omega}$ the map $g_s$ is Lipschitz since if $a \hook k = b \hook k$ then $g_s (a) \hook k + |s| = b \hook k + |s|$. Finally let $f_s:A \to B$ be $\pi \circ f \circ g_s$. Again this is a composition of Lipschitz functions hence Lipschitz. To complete the proof we just have to see that $\bigcup_{s \in \omega^{<\omega}} f_s``A$ covers $B$. Let $b \in B$. Then $b' = \pi^{-1}(b)$ is equal to $f(a')$ for some $a' \in A'$ since $f$ was assumed to be surjective. But now $a' = s^\frown a$ for some $s$ and $a \in A$ by definition and therefore $f_s (a) = b$ as needed. 
\end{proof}

We note that the proof of the above clearly generalizes in the obvious way.
\begin{theorem}
   $\BA_{\rm Lip} (\baire)$ implies that for all $f:\omega \to \omega$ we have $\BA_{\rm Lip}(\Pi_{n < \omega} f(n))$ and $\overline{\BA}_{\rm Lip}(\baire)$ implies $\overline{\BA}_{\rm Lip}(\Pi_{n < \omega} f(n))$ where these axioms are defined in analogue to the case $f(n):= 2$ for all $n < \omega$. 
\end{theorem}
In fact it is more generally true that if $f, g \in (\omega + 1)^\omega$ with $f$ everywhere dominated by $g$ then $\BA_{\rm Lip}(\Pi_{n < \omega} g(n))$ implies $\BA_{\rm Lip}(\Pi_{n < \omega} f(n))$ where these are defined in the obvious way.

Moving on to some other applications, recall the following, which was also mentioned in the introduction:

\begin{fact}
    Assume $\BA$.
    \begin{enumerate}
        \item (Abraham-Rubin-Shelah \cite[Theorem 7.1]{ARS85}) $\cc = 2^{\aleph_1}$.

        \item (Todor\v{c}evi\'{c}, see \cite{Todorcevic89}) $\mfb > \aleph_1$
    \end{enumerate}
\end{fact}

The arguments of \cite[Section 3]{weakBA} easily show that both items are consequences of $\overline{\BA}_{\rm Lip}(\cantor)$. The second of these is in fact a consequence of $\BA_{\rm Lip}(2^\omega)$ - a fact we prove below. We conjecture the first is as well but can only prove this with slightly more assumptions. 

\begin{lemma} \label{cardinal arithmetic}
    If $\cc < \aleph_{\omega_1}$ then the axiom $\BA_{\rm Lip}(2^\omega)$ implies $\cc = 2^{\aleph_1}$.
\end{lemma}

We need the following fact. Recall that for a cardinal $\kappa$ two subsets $A, B \subseteq \kappa$ of size $\kappa$ are {\em almost disjoint} if $|A \cap B| < \kappa$. A family $\mathcal A \subseteq P(\kappa)$ is {\rm almost disjoint} if its members are pairwise so and they all have size $\kappa$. We call such a family {\em almost disjoint on} $\kappa$ when we want to specify the cardinal.

\begin{fact}[Baumgartner, see page 414 of \cite{ADBaum}]
    If $\cc < \aleph_{\omega_1}$ then there is an almost disjoint family on $\omega_1$ of size $2^{\aleph_1}$. 
\end{fact}

Given this fact we can present the proof of Lemma \ref{cardinal arithmetic}.

\begin{proof}[Proof of Lemma \ref{cardinal arithmetic}]
    Fix an $\aleph_1$-dense $A \subseteq 2^\omega$ and fix an almost disjoint family on $\omega_1$ of size $2^{\aleph_1}$, say $\mathcal X$. The existence of such is the only application of the assumption $\cc < \aleph_{\omega_1}$. Suppose towards a contradiction that moreover $2^{\aleph_0} < 2^{\aleph_1}$. Partition $A$ into $\omega_1$ many countable dense subsets $\{A_\alpha\}_{\alpha \in \omega_1}$ i.e. for $\alpha \neq \beta$ then $A_\alpha \cap A_\beta = \emptyset$ and $A = \bigcup_{\alpha \in \omega_1} A_\alpha$ (the existence of such a partition is possible by $\aleph_1$-density).  For each $X \in \mathcal X$ let $A_X = \bigcup_{\alpha \in X} A_\alpha$. This is also $\aleph_1$-dense. Note that since $X, Y \in \mathcal X$ are almost disjoint we have that $A_X \cap A_Y$ have countable intersection as well.

    By $\BA_{\rm Lip}(2^\omega)$ each $X \in \mathcal X$ there is a Lipschitz function $f_X:2^\omega \to 2^\omega$ so that $f_X \hook A$ injects into $A_X$. Since there are only continuum many continuous functions on $2^\omega$ there are distinct $X, Y \in \mathcal X$ for which $f_X = f_Y$. But this is a contradiction since the image of $A$ under $f_X = f_Y$ must have size $\aleph_1$ but gets mapped to two sets whose intersection is countable. 
\end{proof}


Next we show that $\BA_{\rm Lip}(2^\omega)$ has the same effect on the bounding number as $\BA$. 
\begin{lemma}
    $\BA_{\rm Lip}(2^\omega)$ implies $\mfb > \aleph_1$. \label{btheorem}
\end{lemma}

This is a relatively simple consequence of a result of Bartoszy\'nski-Shelah proved in \cite{continuousimagesofsetsofreals}. We recall the relevant fact.

\begin{fact}[Theorem 1 of \cite{continuousimagesofsetsofreals}]
    There is a set $X \subseteq 2^\omega$ of size $\mfb$ so that no continuous $f:X \to \baire$ can map $X$ onto an unbounded set. \label{BSfact}
\end{fact}

\begin{proof}[Proof of Lemma \ref{btheorem}]
  Assume $\BA_{\rm Lip}(2^\omega)$ and, towards a contradiction assume $\mfb = \aleph_1$. Let $B = \{b_\alpha \; | \; \alpha \in \omega_1\}$ be a $\mfb$-scale i.e. a set which is unbounded and so that if $\alpha < \beta$ then $b_\alpha \leq^* b_\beta$. The assumption on $\mfb$ guarantees the existence of such a set. Since $2^\omega$ can be written as $\baire \cup C$ with $C$ a (disjoint) countable dense set, we can treat $B \subseteq 2^\omega$ and in fact assume $B$ is $\aleph_1$-dense. Let $X \subseteq 2^\omega$ be a set of size $\aleph_1$ as in Fact \ref{BSfact} and let $A \supseteq X$ be $\aleph_1$-dense\footnote{In fact an inspection of the proof of \cite[Theorem 1]{continuousimagesofsetsofreals} shows that $X$ itself can be assumed to be $\aleph_1$-dense but we do not need this fact.}. Let $f:X \to B$ be a continuous, injective function. Of course our assumption of $\BA_{\rm Lip}(2^\omega)$ gives that $f$ is in fact Lipschitz but this is not even needed. We now immediately have a contradiction however to the defining property of $X$ since any $\aleph_1$-sized subset of $B$, in particular the image of $X$ under $f$, will be again unbounded. 
\end{proof}

\begin{remark}
    The above clearly has nothing to do with Lipschitz functions. Indeed it is enough to know that for every $A, B \subseteq 2^\omega$ of size $\aleph_1$ there is a continuous mapping from $A$ to $B$ with uncountable range. 
\end{remark}

\section{Consistency}
We now present the proof of the consistency of all of the Baumgartner Lipschitz axioms. A standard bookkeeping argument shows it is enough to prove the following lemma. 

\begin{lemma}
    Assume $\CH$. Let $A, B \subseteq \baire$ be $\aleph_1$-dense. There is a ccc partial order $\P_{A, B}$ forcing the existence of a Lipschitz bijection $\dot{f}:A \to B$. \label{consistencylemma}
\end{lemma}

The proof of this is very similar to the analgous result in \cite{Baum73}. We have in particular tried to follow the notation used there so as to help the reader familiar with that paper follow the proof. 
    \begin{proof}
    Assume $\CH$. Fix $A$ and $B$ which are $\aleph_1$-dense in $\baire$ and let $\P'$ be the set of finite partial Lipschitz injections $f:A \to B$. In other words $p \in \P$ is a finite partial injection from $A$ to $B$ so that for all $a, a' \in {\rm dom}(p)$ if $a \hook k = a' \hook k$ then $p(a) \hook k = p(a') \hook k$ for all $k < \omega$. We will find a ccc suborder of $\P'$. First, let $S$ be the set of all $x$ which are a finite set of pairs $(s_i, t_i) \in (\omega^{<\omega})^2$ so that every $(s_i, t_i) \in x$ is in $(\omega^k)^2$ for some $k < \omega$ and $x$ is {\em separating} i.e. if $(s, t), (u, v) \in x$ then $s$ and $u$ are incomparable and similarly for $t$ and $v$. Equivalently $[s] \cap [u] = \emptyset$ and the same for $v$ and $t$. Note that every $x \in S$ is essentially a basic open around some $p \in \P'$ (viewed as an element of $(\baire)^{|p|}$) which is refined enough that no two elements in the domain (respectively the range) are in the same ball. Let $\P'(x)$ be the set of all $p \in \P'$ so that $|p| = |x|$ and for each $a \in {\rm dom}(p)$ there is a unique $(s, t) \in x$ with $(x, p(x)) \in [s] \times [t]$. Of course there are many $x$ so that $p \in \P'(x)$ for each $p \in \P'$. 

    Using $\CH$, enumerate all countable subsets of $\P '$ which are contained in some $\P ' (x)$ as $\{d_\alpha \; | \; \alpha < \omega_1\}$. For each such, let $x_\alpha$ be such that $d_\alpha \subseteq \P'(x_\alpha)$. Let $c_\alpha$ be the set of all $p \in \P'$ so that for every $x \in S$ if $p \in \P'(x)$ then there is $q \in d_\alpha \cap \P'(x)$. Note that $c_\alpha$ is in essence the closure of $d_\alpha$ in $\P '$ in the sense that it is literally the topological closure of $d_\alpha$ in $(\baire)^{2|x|}$ intersected with the $|x|$-length conditions in $\P '$. 

    \begin{claim}
        Suppose $x \in S$ and $U \subseteq \P'(x)$. There is an $\alpha < \omega_1$ so that $d_\alpha \subseteq U \subseteq c_\alpha$. 
    \end{claim}

\begin{proof}
    Since $U \subseteq (\baire)^{2|x|}$ it is a separable metric space hence it has a countable dense set $d = d_\alpha$ for some $\alpha$. Now, density ensures that $U \subseteq c_\alpha$ by the definition of $c_\alpha$.  
    \end{proof}
We note the interesting case of this is when $U$ is uncountable. 

    We can now construct the partial order $\P$. We will find partitions $A = \bigcup_{\alpha < \omega_1} A_\alpha$ and $B = \bigcup_{\alpha < \omega_1} B_\alpha$ so that $\P = \{p \in \P' \; | \; \forall \alpha \, p\hook A_\alpha:A_\alpha \to B_\alpha\}$ is ccc and each $A_\alpha$ and $B_\alpha$ are countable and dense. In fact, we don't need to assume anything about the $B$ partition\footnote{This asymmetry typifies the difference between the Lipschitz case and the linear order case. The inverse of a Lipschitz function need not be Lipschitz while the inverse of a linear order isomorphism is again a linear order isomorphism.} so let $B = \bigcup_{\alpha < \omega_1} B_\alpha$ be a partition of $B$ into countable dense sets and for each $\alpha < \omega_1$ enumerate $B_\alpha = \{b^\alpha_i\; | \; i < \omega\}$. Suppose that $\beta < \omega_1$ and for all $\alpha < \beta$ we have already defined $A_\alpha$. We work on defining $A_\beta$. 

We need to define countable, dense $\{a^\beta_n\; | \; n < \omega\} \subseteq A \setminus \bigcup_{\alpha < \beta} A_\alpha$. Let $a^\beta_0$ be the least element of $A$ not in any $A_\alpha$ so far (relative to some fixed enumeration in order type $\omega_1$). We do this simply to ensure the construction covers $A$. Let $\{s_n\; | \; n < \omega\}$ enumerate $\omega^{<\omega}$ so that $a^\beta_0 \in [s_0]$ and for every $i < \omega$ we have $b^\beta_i \in [s_i]$ (we can choose any enumeration we like so we can do this - at worst we can let $s_0 = [\emptyset]$). Suppose now that $\{a^\beta_i\; | \; i<n + 1\}$ have been defined and $a_i^\beta \in [s_i]$. We define $a^\beta_n$ as follows. For each $p$ in the already constructed part of $\P$, $i < n+1$, $\alpha < \beta$, $(s, t) \in x_\beta$, let $X_{i, p, \beta, (s, t)}$ be the following set $$\{a \in A \cap [s]\; | \; p \cup (a, b^\beta_i) \in c_\beta \, \land b^\beta_i \in [t]\}$$ if this set is countable and let it be the empty set otherwise. Now choose $a^\beta_n$ to be any element of $[s_n]$ in $A$ not already constructed nor in any of the $X_{i, p, \beta, (s, t)}$'s. Note that by $\aleph_1$-density this is possible.

    Let $\P$ be the set of $p \in \P'$ so that for all $\alpha < \omega_1$ we have $p\hook A_\alpha$ maps into $B_\alpha$. Now we need to check that the construction works - i.e. that $\P$ is ccc. First we make a slight technical reduction. Note that $\P$ has a dense subset consisting of elements $p$ so that if $\beta$ is greatest with $p \cap A_\beta \times B_\beta$ non-empty then there is a unique $a \in {\rm dom}(p) \cap A_\beta$ and if $a = a^\beta_k$ and $p(a) = b^\beta_l$ then $l < k$. We call this unique element of $A_\beta \times B_\beta$ the {\em last constructed element} of $p$. We work with this dense subset. Towards a contradiction assume that $U \subseteq \P$ is an uncountable antichain. By standard arguments we can assume that there is an $n < \omega$ so that $|p| = n$ for all $p \in U$, this $n$ is minimal for which there is an antichain like this. Also, for each $p \neq q \in U$ we can assume ${\rm dom}(p) \cap {\rm dom}(q) = {\rm range}(p) \cap {\rm range}(q) = \emptyset$. This is because we can first apply a $\Delta$-system argument to the domain of the conditions to ensure that the domains of an uncountable subset of $U$ form a $\Delta$-system and then, apply pigeon hole plus the fact that there are only countably many possible targets for any given $a \in A$ to thin out further and obtain an uncountable antichain where the elements pointwise agree on their (shared, unique) common domain. However now if two such conditions are not compatible then they remain so with this shared piece removed - which would contradict the minimality of $n$ were it not empty. Moreover, we can assume that there is an $\alpha < \omega_1$ so that $d_\alpha \subseteq U \subseteq \P'(x_\alpha)$ and there are $(s, t) \in x_\alpha$ so that every $p \in U$ has its last constructed pair in $[s]\times [t]$ (with respect to the construction given above). Moreover in that last constructed pair we know, because of the dense subset we restricted to, that the domain element is the last constructed one. Let $Q = \{p \setminus [s]\times [t] \; | \; p \in U\}$ and let $B' = B \setminus \bigcup_{p \in d_\alpha} {\rm range}(p)$. 

    \noindent \underline{Case 1:} $n = 1$. In this case we claim that for every $b \in B'$ the set $Z_b$ of all $a$ so that $\{(a, b)\} \in c_\alpha$ is a singleton. Indeed, suppose $a, a' \in Z_b$ and let $k$ be least with $a (k) \neq a'(k)$. Then there are $(a_0, b_0)$ and $(a_0', b_0')$ in $d_{\alpha}$ with $b_0 \hook k + 1 = b_0' \hook k+1 = b \hook k+1$ and $a_0 \hook k + 1 = a \hook k+1$ and $a_0'\hook k+1 = a' \hook k+1$. But then these two conditions are compatible contradicting the assumption on $U$. 

    \noindent \underline{Case 2:} $n > 1$. We need some notation - given a condition $p \in \P$ and a natural number $l < \omega$ let $p \hook l$ denote the set of $\{(a \hook l, p(a) \hook l)\; | \; a \in {\rm dom}(p)\}$. We claim that for all but countably many $q \in Q$ and every $b \in B'$ the set $Z_b$ of $a$ so that $q \cup (a, b) \in c_\alpha$ is countable. Suppose not and let $Q'$ be an uncountable subset of $Q$ with the property above. By the minimality of $n$ we have to be able to find compatible $q_0, q_1 \in Q'$. Let $l$ be large enough so that every element in the ranges of $q_0$ and $q_1$ (and hence also their domains by the Lipschitz property) are separated. Note that now any two $q', r'$ so that $q' \hook l = q_0 \hook l$ and same for $q_1$ and $r'$ will be compatible. Now let $a_0$ and $a_1$ be so that $q_0 \cup (a_0, b) \in c_\alpha$ and the same for $q_1 \cup (a_1, b)$ and without loss of generality we may assume that $a_0 \hook l  \neq a_1 \hook l$ and indeed $l$ separates all the points in $q_i \cup (a_i, b)$. Now let $r_0, r_1 \in d_\alpha$ so that $r_0 \hook l = q_0 \cup (a_0, b) \hook l$ and the same for $r_1$ with $q_1$ and $a_1$. Then again $r_1$ and $r_0$ must be compatible. 

Now it follows from the above two cases that either way we must find $p \in U$ so that the $(a, b) \in [s]\times[t] \cap p$ are in some $A_\beta \times B_\beta$ with $\beta$ above $\alpha$ and moreover for every $b \in B'$ the set $Z_b$ defined above with respect to $p \setminus [s] \times [t]$ is countable. But then $a$ could not be chosen as it was by the construction of the partition. This completes the proof of the ccc.

To finish the proof of Lemma \ref{consistencylemma} we need to check that $\P$ forces the existence of a Lipschitz bijection from $A$ to $B$. Let $G \subseteq \P$ be generic and work in $V[G]$. Clearly $g:= \bigcup G$ is a Lipschitz injection from some subset of $A$ into $B$. We need to show that in fact the domain is all of $A$ and the range is all of $B$. Thus it remains to show that for all $a \in A$ and $a \in B$ there is a dense set of conditions $q$ so that $a \in {\rm dom}(q)$ and $b \in {\rm range}(q)$. Let $a \in A$, $b \in B$ and $p \in \P$. Without loss assume that $a \notin {\rm dom} (p)$ and $b \notin {\rm range}(p)$. Let $\alpha, \beta < \omega_1$ so that $a \in A_\alpha$ and $b \in B_\beta$. By density plus the finiteness\footnote{This assumption separates the $\baire$ case of this forcing notion from the $2^\omega$ case.} of $p$ there are $k < \omega$ so that for no $a' \in {\rm dom}(p)$ do we have $a' (0) = k_0$ and there is a $c \in A_\beta$ so that $c(0) = k$. Set $p' = p \cup \{\langle c, b\rangle\}$. Note that this is still Lipschitz since $c$ differs from every element of the domain of $p$ on the first coordinate. Let $l < \omega$ now be maximal so that there is an $a' \in {\rm dom}(p ')$ with $a \hook l = a' \hook l$. Let $j < \omega$ so that no $b' \in {\rm range}(p')$ is such that $b' (l) = j$. Again by density there is a $d \in B_\alpha$ so that $d \hook l+1 = [p'(a') \hook l]^\frown j$. Let $q = p' \cup \{\langle a, d\rangle\}$. An easy verification shows that this is still Lipschitz and hence a condition. 
    \end{proof}

As a consequence we have the proof of Theorem \ref{consistencyintro} which we repeat below.
    \begin{theorem}\label{consistency}
   $\overline{\BA}_{\rm Lip} (\baire)$ can be forced by a finite support iteration of ccc forcing notions over a model of $\CH$.    
    \end{theorem}

    Using the ``$\CH$ trick" from \cite{BaumPFA} we also have the following.

    \begin{corollary}
        $\PFA$ implies $\overline{\BA}_{\rm Lip} (\baire)$. 
    \end{corollary}
    
    \section{Additivity of the Null Ideal}
In this section we prove that $\overline{\BA}_{\rm Lip}(\cantor)$ and hence $\overline{\BA}_{\rm Lip} (\baire)$ implies that the additivity of the null ideal is larger than $\aleph_1$. Since $\add(\Null)$ is the smallest cardinal in the Cicho\'{n} diagram this proves Theorem \ref{cichonintro} from the introduction. 

\subsection{The Case of $\baire$}
The case of $\baire$ is actually much more straightforward so we present it first as a warmup even though it is a (seemingly) weaker result. 
\begin{lemma}
    $\overline{\BA}_{\rm Lip}(\baire)$ implies $\add(\Null) > \aleph_1$. \label{add}
\end{lemma}
We need the following characterization of the null ideal due to Bartoszy\'nski. Recall that if $h:\omega \to \omega$ is strictly increasing then an $h$-{\em slalom} is a function $\varphi:\omega \to [\omega]^{<\omega}$ so that for all $n$ we have $|\varphi(n)| \leq h(n)$. If we omit the $h$ then it is implied that ``slalom" means $2^{n+1}$-slalom. Say that a function $f\in \omega^\omega$ is caught by a slalom $\varphi$, in symbols $f \in^* \varphi$ if for all but finitely many $n$ we have $f(n) \in \varphi(n)$. Similarly let us write $f \in \varphi$ if for every $n < \omega$ we have $f(n) \in \varphi(n)$. Finally for a set $A \subseteq \baire$ we say a slalom $\varphi$, {\em captures} $A$ if it eventually captures every element. 

\begin{fact}[Bartoszy\'nski]\label{slalomfact}
    Let $h:\omega \to \omega$ be strictly increasing. For any cardinal $\kappa$ the following are equivalent.\begin{enumerate}
        \item $\kappa < {\rm null}(\mathcal N)$
         \item For every $A \subseteq \baire$ of size $\kappa$ there is an $h$-slalom that eventually captures $A$. 
    \end{enumerate}
\end{fact}

Note the point is that the cardinal doesn't depend on which $h$ we choose - however it must be uniform for all $A$ of size ${<}\kappa$. 

\begin{proof}[Proof of Lemma \ref{add}]
    Assume $\overline{\BA}_{\rm Lip}(\baire)$. We will show that every set of size $\aleph_1$ is caught in an $h$-slalom for $h(n) = n2^{n+1}$. Let $A$ be an arbitrary set of set $\aleph_1$. By possibly making it bigger we can assume that $A$ is $\aleph_1$-dense. 
    Let $B \subseteq \baire$ defined as follows. For each $s \in \omega^{<\omega}$ let $B_s \subseteq [s]$ be an $\aleph_1$-sized set of $x \supseteq s$ so that if $k > {\rm dom}(s)$ then $x(k) = 0$ or $x(k) = 1$. Let $B = \bigcup_{s \in \omega^{<\omega}} B_s$. In short, $B$ is an $\aleph_1$-dense set of functions which are eventually bounded by $2$.

    By assumption there is a Lipschitz $f:\baire \to \baire$ so that $f``B = A$. Fix $s \in \omega^{<\omega}$ and let $\varphi_s:\omega \to [\omega]^{<\omega}$ be defined by $\varphi(n) = \{m \; | \; \exists x \in B_{s} \; f(x)(n) = m\}$. 

    \begin{claim}
        This is an $h$-slalom for $h:\omega \to \omega$ defined by $h(n) = 2^{n+1}$ for all $n < \omega$.
    \end{claim}

    \begin{proof}
        Let $|s| = n_s$. For each $n_s < n < \omega$ there are only $2^{n+1 - n_s}$ many $t \in \omega^{n+1}$ so that $x \hook (n+1) = s$ for each $x \in B_{s}$ since each such $x$ is bounded by $2$ above the length of $s$. By the defining property of the Lipschitz function that means there are at most $2^{n+1 - n_s}$ many $s \in \omega^{n+1}$ so that $s = f(x) \hook n+1$ and hence there are only at most $2^{n+1 - n_s}$ many $m$ so that $\exists x \in B_{s} \; f(x)(n) = m$.
    \end{proof}
    
Now observe that if $x \in B_{s}$ then for every $n < \omega$ we have $f(x)(n) \in \varphi_s(n)$ by construction. In other words, for each $s \in \omega^{<\omega}$ the forward image $f``B_s$ is caught (totally, not eventually) by $\varphi_s$. In particular there are countably many slaloms $\{\varphi_s\; | \; s \in \omega^{<\omega}\}$ so that every element of $A$ is totally caught by (at least) one of them. Now enumerate $\omega^{<\omega}$ as $\{s_n\; | \; n < \omega\}$ and let $\varphi ( n) =\bigcup_{i < n} \varphi_{s_i} (n)$. This is a $n2^{n+1}$-slalom which eventually captures every element of $A$, completing the proof.    
\end{proof}

\subsection{The Case of $\cantor$}
We now turn to the harder, but (apparently) stronger theorem that the conclusion of Theorem \ref{add} follows with $\baire$ replaced by $\cantor$.
\begin{theorem}
    $\overline{\BA}_{\rm Lip}(\cantor)$ implies $\add(\Null) > \aleph_1$. \label{add2}
\end{theorem}
Recall that a set $A \subseteq 2^\omega$ is {\em null additive} if for every null set $X \subseteq 2^\omega$ the collection of translations $A + X:= \{a + x\; | \; a \in A \, x \in X\}$ is null. Shelah proved in \cite{nulladditive} a combinatorial characterization of null additivity that is not dissimilar to that of Fact \ref{slalomfact}. Before giving this let us recall the relevant definitions. 
\begin{definition}
    A {\em corset} is a function $c:\omega \to \omega \setminus \{0\}$ which is non decreasing and converges to infinity. Given a tree $T \subseteq 2^{< \omega}$ and a corset $c$, we say that $T$ has {\em width $c$} if for all $n < \omega$ $|2^n \cap T| \leq c(n)$. 
\end{definition}

We now state Shelah's theorem.
\begin{theorem}[See Theorem 13 of \cite{nulladditive}] \label{shelahnulladditive}
   For every $X \subseteq \cantor$ the following are equivalent.
   \begin{enumerate}
       \item $X$ is null additive.
       \item For every corset $c$ there are trees $S_m$, $m < \omega$ of width $c$ so that $X \subseteq \bigcup_{m \in \omega} [S_m]$. 
   \end{enumerate}
\end{theorem}

We aim to show first the following. 
\begin{lemma}
    If $\overline{\BA}_{\rm Lip}(2^\omega)$ holds then every subset of $2^\omega$ of size $\aleph_1$ is null additive. 
\end{lemma}

\begin{proof}
    Clearly it suffices to show the result for $\aleph_1$-dense sets. Fix an $\aleph_1$-dense set $B \subseteq \cantor$ and a corset $c:\omega\to \omega$. Enumerate $2^{<\omega}$ as $\{t_n\mid n < \omega\}$ and for each $n < \omega$ fix a tree $T_n$ of width $c$ contained in $[t_n]$. For each $n < \omega$ pick a set $A_n \subseteq [T_n]$ which is $\aleph_1$-dense in the space $[T_n]$. Let $A = \bigcup_{n < \omega} A_n$. Clearly $A$ is $\aleph_1$-dense an can be covered by $\bigcup_{n < \omega}[T_n]$. First, let $f:\cantor \to \cantor$ be an arbitrary Lipschitz function. By Lemma \ref{homomorphism representation} this is induced by a homomorphism $\tilde{f}:2^{<\omega} \to 2^{< \omega}$. Moreover given any tree $T \subseteq 2^{<\omega}$ the image of the restriction of this homomorphism to $T$ is again a tree.

    \begin{claim}
        If $T \subseteq 2^{<\omega}$ is a tree of width $c$ and $g:T \to 2^{<\omega}$ is a homomorphism then ${\rm Im}(g)$ has width $c$ as well.
    \end{claim}

    \begin{proof}
        For each $n < \omega$ we have by assumption that $g\hook T \cap 2^n$ is a finite function with domain a set of size at most $c(n)$ hence its range also has size at most $c(n)$. 
    \end{proof}
    
       Now, applying $\overline{\BA}_{\rm Lip}(\cantor)$ we have countably many Lipschitz functions $f_n:\cantor \to \cantor$ so that $B = \bigcup_{n < \omega} f_n``A$. For each $n, m < \omega$ therefore, by the claim we have that $S_{n, m}:= \tilde{f}_m``T_n$ is a tree of width $c$. But now $B \subseteq \bigcup_{n, m \in \omega} [S_{n, m}]$ and therefore we have found countably many trees of width $c$ whose branches cover $B$. Since $c$ was arbitrary and did not depend on $B$, it must be the case that $B$ can be covered by countably many trees of {\em any} fixed corset width. By Theorem \ref{shelahnulladditive} this completes the proof.
\end{proof}

Theorem \ref{add2} now follows as a corollary since, by \cite[Theorem 2.7.13]{BarJu95} we have that $\add(\Null) = {\rm min}\{\mfb, \add^*(\Null)\}$ where $\add^*(\Null)$ is the minimum size of a non null additive set. Since we already know that $\BA_{\rm Lip}(\cantor)$ implies $\mfb > \aleph_1$ it follows $\add(\Null) > \aleph_1$ as well. The interest in this result - besides being an application of our axioms is the following.
\begin{corollary}
    $\mfp > \aleph_1$ does not imply $\overline{\BA}_{\rm Lip}(\cantor)$ in particular $\BA(\cantor)$ does not imply its Lipschitz version. 
\end{corollary}
This is simply because $\mfp > \aleph_1$ does not imply even the covering number of $\Null$ let alone $\add(\Null)$ is greater than $\aleph_1$. A stronger result along these lines will be proved in the next section.

\section{Avoiding Sets and Fragments of $\MA$}

In this final section we consider whether the analogue of the main result of \cite{AvrahamShelah81} holds for the Lipschitz axioms - i.e. does $\MA$ imply the failure of any of the Lipschitz axioms. While this is still open we have the following partial result which states that even the weakest $\BA_{\rm Lip}(2^\omega)$ is not implied by a large fragment of $\MA$. 

\begin{theorem}
    $\BA_{\rm Lip}(2^\omega)$ is not implied by $\MA ({\rm Knaster}) \, +$ ``every (wide) Aronszajn tree is special". \label{noimplicationofMA}
\end{theorem}

The rest of this section is devoted to the proof of this theorem. As a first step we recall some definitions.
\begin{definition}
    A partial order $\P$ is {\em Knaster} if for every uncountable $A \subseteq \P$ there is an uncountable $B \subseteq A$ of pairwise compatible elements.
\end{definition}

Let $T$ be a tree of height $\omega_1$ with no cofinal branch. Recall that $T$ is called {\em special} if there is a function $f:T \to \omega$ which is injective on chains.
\begin{definition}
    Let $T$ be a tree of height $\omega_1$. The {\em specializing forcing} for $T$, denoted $\P_T$ is the forcing notion consisting of finite partial functions $p:T \to \omega$ which is injective on chains.
\end{definition}
This forcing was first introduced in \cite{BMR70} to show that $\MA$ implies that every tree of height $\omega_1$ with no cofinal branch and cardinality less than $\cc$ is special. The key fact is the following.
\begin{fact}
    $P_T$ is ccc if and only if $T$ has no cofinal branch.
\end{fact}
The specializing forcing notions $\P_T$ were studied in more detail in \cite{specialtree}. There it is shown, among other things the following.
\begin{theorem}[Theorem 3.6 of \cite{specialtree}]
    If $T$ is a tree of height $\omega_1$ with no cofinal branch and $\dot{x}$ is a $\P_T$ name for a real then there is a $\sigma$-centered subforcing $\Q_{\dot{x}}$ so that $\dot{x}$ is forced to be in $V^{\Q_{\dot{x}}}$. \label{theoremspecial}
\end{theorem}

The iteration used to produce Theorem \ref{noimplicationofMA} will be the obvious one - a finite support iteration of Knaster forcing notions and forcing notions of the form $\P_T$ using standard bookkeeping devices as found in e.g. \cite[Chapter 5]{KenST}. The point is that we will start in the ground model with a pair of $\aleph_1$-dense subsets of $\cantor$ which are counterexamples to $\BA_{\rm Lip}(\cantor)$ and will remain so by an iteration theorem we prove below. This pair of $\aleph_1$-dense sets are given by the following combinatorial property.

\begin{definition}
    Let $A, B \subseteq \baire$. We say that $A$ {\em Lipschitz avoids} $B$ if there is no uncountable $Z \subseteq A$ which can be injected into $B$ by a Lipschitz function.
\end{definition}

This definition is similar to the more general notion of avoiding introduced in \cite{weakBA}. Clearly if $A$ and $B$ are $\aleph_1$-dense subsets of $\cantor$ and $A$ Lipschitz avoids $B$ then in particular $A$ and $B$ witness the failure of $\BA_{\rm Lip} (\cantor)$. It is easy to build a pair of $\aleph_1$-dense $A, B \subseteq \cantor$ under $\CH$ so that $A$ Lipschitz avoids $B$. Similarly the following was also shown\footnote{Actually a stronger fact was proved but since we will not preserve the stronger property we do not bother to mention it.} in \cite{weakBA}.
\begin{fact}
    If $\kappa$ is an uncountable cardinal and $\P$ is the standard forcing to add either $\kappa$ many random reals or $\kappa$ many Cohens then any partition of the generic reals into two pieces $A$ and $B$ will Lipschitz avoid one another. 
\end{fact}

The justification for this definition is that it plays well with iterations. The following is a variation of the iteration theorem given in \cite[Theorem 5.9]{weakBA}.
\begin{lemma}
    Let $\delta$ be an ordinal, $A, B \subseteq \baire$ and assume $A$ Lipschitz avoids $B$. If $\langle \P_\alpha , \dot{\Q}_\alpha\; | \; \alpha \in \delta\rangle$ is a finite support iteration of ccc forcing notions so that for each $\alpha < \delta$ we have that $\forces_\alpha$``$\dot{\Q}_\alpha$ preserves that $\check{A}$ Lipschitz avoids $\check{B}$", then $\forces_\delta$ ``$\check{A}$ Lipschitz avoids $\check{B}$". \label{iteration}
\end{lemma}

\begin{proof}
    Fix $A, B \subseteq \omega^\omega$ so that $A$ Lipschitz avoids $B$ and let $\langle \P_\alpha , \dot{\Q}_\alpha\; | \; \alpha \in \delta\rangle$ is a finite support iteration of ccc forcing notions so that for each $\alpha < \delta$ we have that $\forces_\alpha$``$\dot{\Q}_\alpha$ preserves that $\check{A}$ Lipschitz avoids $\check{B}$". We need to show that $\P_\delta$ forces that $A$ Lipschitz avoids $B$. The proof is by induction on $\delta$. The successor case is by assumption so we check the limit case. Assume towards a contradiction that $\forces_\delta$``$\dot{Z} \subseteq \check{A}$ is uncountable and $\dot{f}:\dot{Z} \to \check{B}$ is Lipschitz with uncountable range. There are two subcases. 

    \noindent \underline{Case 1}: $\delta$ has countable cofinality. Let $\delta_n \nearrow \delta$ be a strictly increasing sequence of ordinals with $n \in \omega$. Let $G_\delta$ be $\P_\delta$-generic and for each $n < \omega$ let $G_n = G_\delta \cap \P_{\delta_n}$ be the $\P_{\delta_n}$-generic given by the subforcing. Work momentarily in $V[G_\delta]$. Let $f_n$ be the partial function from $X$ to $Y$ given by $f_n(x) = y$ if and only if there is a condition $p \in G_n$ forcing that $\dot{f}(\check{x}) = \check{y}$. Note that by the finite support we have that $$\forces_{\delta} \dot{f}^{G_\delta} =\bigcup_{n < \omega} f_n$$. Moreover observe that $f_n \in V[G_n]$, where, by inductive assumption, $A$ still Lipschitz avoids $B$. By pigeonhole is an $n<\omega$ so that $f_n$ has domain uncountable. But being Lipschitz is absolute so this $f_n$ is already a counterexample to avoiding, which is a contradiction. 

    \noindent \underline{Case 2}: $\delta$ has uncountable cofinality. In $V[G_\delta]$ by Lemma \ref{lifting lipschitz} we can lift $\dot{f}^{G_\delta}$ to a function $\hat{f}$ whose domain is $\bar{Z}$. Since $\bar{Z}$ is closed, both $\hat{f}$ and $\bar{Z}$ are coded by reals hence by the ccc plus finite support there is a $\gamma < \delta$ so that $\bar{Z}, \hat{f} \in V[G_\gamma]$. Work in $V[G_\gamma]$ for such a $\gamma$ and note by assumption that $A$ still Lipschitz avoids $B$ in $V[G_\gamma]$. Observe first that if there are $x \in A$ and $y \in B$ so that some condition $p \in \P_{\gamma, \delta}$ forces that $\dot{f}(\check{x}) = \check{y}$ then by Schoenfield absoluteness in $V[G_\gamma]$ we already have that $\hat{f}(x) = y$. Moreover in $V[G_\delta]$ we have that $\hat{f}$ is Lipschitz - but this is also absolute since to be Lipschitz is a $\Pi^1_2$ statement with the parameters $\hat{f}$ and $\bar{Z}$: $\forall x, y \in \bar{Z} \, \forall k \in \omega\; ( x \hook k = y \hook k \to \hat{f}(x) \hook k = \hat{f} (x) \hook k)$. Therefore in $V[G_\gamma]$ we have that $\hat{f}$ is Lipschitz and its domain includes uncountably many elements of $A$ which are getting mapped into $B$. Therefore $\hat{f}\hook (A \cap \hat{f}^{-1}(B))$ is a Lipschitz map with uncountable range from a subset of $A$ into $B$, contradicting avoiding. 
\end{proof}

\begin{lemma}\label{knaster}
    Suppose $\P$ is Knaster and $A, B \subseteq 2^\omega$ is so that $A$ Lipschitz avoids $B$ then $\P$ forces that $A$ Lipschitz avoids $B$. 
\end{lemma}

\begin{proof}
    Suppose towards a contradiction that $p \in \P$ forces that $\dot{Z} \subseteq A$ has size $\omega_1$ and there is a $\dot{f}:\dot{Z} \to B$ which is injective and Lipschitz. Concretely let $\{\dot{z}_\alpha \; | \; \alpha < \omega_1\}$ be names for the elements in some enumeration. For each $\alpha  < \omega_1$ let $p_\alpha \leq p$ be a condition so that there is some $a_\alpha \in A$ and $b_\alpha \in B$ with $$p_\alpha \forces \check{a}_\alpha = \dot{z}_\alpha \; \land \dot{f}(\check{a}_\alpha) = \check{b}_\alpha$$
By the Knaster property there is an uncountable $K \subseteq \omega_1$ so that the set of all $\{p_\alpha \; | \; \alpha \in K\}$ are pairwise compatible\footnote{This of course includes the case where the set of all $p_\alpha$ is actually countable i.e. some particular $q$ is equal to $p_\alpha$ for uncountably many $\alpha$.}. Let $g_K:\{a_\alpha \; | \; \alpha \in K\} \to \{b_\alpha\; | \; \alpha \in K\}$ be defined by $g_K (a_\alpha) = b_\alpha$ for all $\alpha \in K$. By compatibility of the conditions this is a Lipschitz injection from an uncountable subset of $A$ into $B$ contradicting avoiding.
\end{proof}

\begin{lemma}\label{PT}
    Suppose $T$ is a tree of height $\omega_1$ with no cofinal branch. If $A, B \subseteq \baire$ and $A$ Lipschitz avoids $B$ then $\P_T$ forces that $A$ Lipschitz avoids $B$.
\end{lemma}

\begin{proof}
    Fix $T$ as in the statement of the lemma. Suppose $\dot{f}$ is a $\P_T$-name for a Lipschitz function from an uncountable subset of $A$ in $B$ with uncountable range. Like in the proof of Lemma \ref{iteration}, Case 2, in the generic extension we can lift $\dot{f}$ to a $\hat{f}$ which is Lipschitz and has closed domain. By Theorem \ref{theoremspecial}, there is a $\sigma$-centered subforcing of $\P_T$, say $\Q$ which adds the closed set and the Lipschitz function on it. However in this subextension we have that $A$ Lipschitz avoids $B$ since $\sigma$-centered forcing notions are in particular Knaster and we can therefore apply Lemma \ref{knaster}. We are now though in exactly the same situation as in Case 2 of Lemma \ref{iteration} and the same contradiction occurs. 
\end{proof}

We can now prove Theorem \ref{noimplicationofMA}. It is a straightforward consequence of what we have shown so far.

\begin{proof}[Proof of Theorem \ref{noimplicationofMA}]
    Start in a model of $\CH$ and fix $A$ and $B$ which are $\aleph_1$-dense sets of $\cantor$ so that $A$ Lipschitz avoids $B$. By Lemmas \ref{iteration}, \ref{knaster}, and \ref{PT} it follows that the standard iteration for forcing $\MA ({\rm Knaster}) +$``every Aronszajn tree is special" will force these facts and $A$ will still Lipschitz avoid $B$ hence completing the proof. 
\end{proof}

We note oddly that in the model described in the proof above even though no uncountable subset of $A$ can be Lipschitz injected into $B$ in fact $A$ and $B$ are homeomorphic - and even so by a homeomorphism which lifts to one of all of $\cantor$ by fact that $\mfp > \aleph_1$ implies that $\BA (\cantor)$ holds. 

\section{Open Questions}
We finish the paper by collecting some open questions which have appeared throughout the text. The first of these is whether there are more implications between the four axioms we have been discussing.
\begin{question}
    Are there further implications between $\BA_{\rm Lip} (\baire)$, $\BA_{\rm Lip} (\cantor)$, $\overline{\BA}_{\rm Lip} (\baire)$, and $\overline{\BA}_{\rm Lip} (\cantor)$ than those given in Theorem \ref{implications}? In particular are they all equivalent? Do the weak and strong versions equate? Do the Baire and Cantor versions equate?
\end{question}

We can also ask whether the implications of the stronger Lipschitz axioms follow from the weaker ones.
\begin{question}
    Do either $\BA_{\rm Lip} (\baire)$ or $\BA_{\rm Lip} (\cantor)$ imply $\cc = 2^{\aleph_1}$? Do either of $\BA_{\rm Lip} (\baire)$ or $\BA_{\rm Lip} (\cantor)$ imply $\add(\Null) > \aleph_1$?
\end{question}

Similarly we note that we still have not resolved the analogue of the question asked in \cite{Stepranswatson87} regarding $\mfp$. 

\begin{question}
    Do any of the Lipschitz axioms imply $\mfp > \aleph_1$? 
\end{question}

We would also like to know whether $\MA$ itself, and not just the large fragment considered in the previous section, does not imply the Lipschitz axioms.
\begin{question}
    Does $\MA$ imply $\BA_{\rm Lip}(\cantor)$?
\end{question}
Finally floating in the background is the question of the relation between $\BA$ and the axioms considered in this paper.
\begin{question}
    Are there provable relations between $\BA$ and the Lipschitz variations?
\end{question}

\bibliographystyle{plain}
\bibliography{Lipschitz}
\end{document}